\documentclass[12pt, reqno, a4paper]{amsart}
\usepackage{graphicx, epsfig, psfrag}
\usepackage{amsfonts,amsmath,amssymb,amsbsy,amsthm}
\usepackage{bm}
\usepackage{xcolor}
\usepackage{float}
\usepackage{hyperref}
\usepackage{mathrsfs}
\usepackage[top=3cm, bottom=3cm, left=3cm, right=3cm]{geometry}

\usepackage{pdfsync}

\usepackage{environments}

\newcommand{\smfrac}[2]{{\textstyle \frac{#1}{#2}}}

\newcommand{\mymat}[1]{\left[ \begin{matrix} #1 \end{matrix} \right]}

\def\XXint#1#2#3{{\setbox0=\hbox{$#1{#2#3}{\int}$ }
\vcenter{\hbox{$#2#3$ }}\kern-.6\wd0}}

\def\b{\big}
\def\B{\Big}
\def\bg{\bigg}

\def\sep{\,|\,}
\def\bsep{\,\b|\,}

\def\R{\mathbb{R}}
\def\N{\mathbb{N}}
\def\Z{\mathbb{Z}}

\def\dx{\,{\rm d}x}

\def\ds{\,{\rm d}s}

\def\pp{\partial}

\def\<{\langle}
\def\>{\rangle}

\def\del{\delta}
\def\ddel{\delta^2}

\def\ol{\overline}

\def\mA{{\sf A}}

\def\mF{{\sf F}}

\def\mQ{{\sf Q}}

\def\D{\partial}

\def\a{{\rm a}}
\def\Ea{E^\a}
\def\cb{{\rm cb}}
\def\Ecb{E^\cb}
\def\scb{{\rm scb}}
\def\Escb{E^\scb}
\def\x{\ell}

\def\OO{\mathcal{O}}
\def\Err{{\rm Err}}
\def\L{\Lambda}
\def\Ng{\mathcal{N}}
\def\Om{\Omega}
\def\Th{\mathcal{T}_h}
\def\Tm{\mathcal{T}_\a}
\def\PI{{\rm P}_1}
\def\PO{{\rm P}_0}
\def\per{\#}

\begin{document}

\title[An Analysis of Surface Relaxation in the SCB Model]{An Analysis
  of Surface Relaxation in \\ the Surface Cauchy--Born Model}

\author{K. Jayawardana}
\address{K. Jayawardana \\ Department of Mathematics \\
  University College London \\
  Gower Street \\
  London WC1E 6BT \\ UK} 
\email{k.guruge@ucl.ac.uk}

\author{C. Mordacq}
\address{C. Mordacq}
\email{Christelle.Mordacq@gmail.com}

\author{C. Ortner}
\address{C. Ortner\\ Mathematics Institute \\ Zeeman Building \\
  University of Warwick \\ Coventry CV4 7AL \\ UK}
\email{christoph.ortner@warwick.ac.uk}

\author{H. S. Park}
\address{H. S. Park\\  Boston University \\
Department of Mechanical Engineering \\
730 Commonwealth Avenue, ENA 212 \\
Boston \\ MA 02215 \\ USA}
\email{parkhs@acs.bu.edu}

\date{\today}

\thanks{KJ and CM were supported by undergraduate vacation bursaries
  at the Oxford Centre for Nonlinear PDE. CO was supported by the
  EPSRC Grant EP/H003096 ``Analysis of Atomistic-to-Continuum Coupling
  Methods''. HP was supported by NSF grants CMMI-0750395 and
  CMMI-1036460.}



\begin{abstract}
  The Surface Cauchy--Born (SCB) method is a computational multi-scale
  method for the simulation of surface-dominated crystalline
  materials. We present an error analysis of the SCB method, focused on
  the role of surface relaxation. 

  In a linearized 1D model we show that the error committed by the SCB
  method is $\OO(1)$ in the mesh size; however, we are able to
  identify an alternative ``approximation parameter'' --- the
  stiffness of the interaction potential --- with respect to which the
  error in the mean strain is exponentially small. Our analysis
  naturally suggests an improvement of the SCB model by enforcing
  atomistic mesh spacing in the normal direction at the free boundary.
\end{abstract}

\maketitle

\section{Introduction}

Miniaturization of materials to the nanometer scale has led to
unexpected and often enhanced mechanical properties that are not
found in corresponding bulk materials
\cite{cuenotPRB2004,seoNL2011}.  The size-dependence of the mechanical
behavior and properties has been experimentally observed to begin
around a scale of about 100 nanometers~\cite{parkMRS2009}.  A fully
atomistic simulation of a nanostructure of this size would require on
the order of $10^8$ atoms, which motivates the need for
computationally efficient multiscale methods.

The underlying cause for the size-dependent mechanical properties is
that surface atoms have fewer bonding neighbours, or a coordination
number reduction, as compared to atoms that lie within the material
bulk. This results in the elastic properties of surfaces being
different from those of an idealized bulk material~\cite{parkMRS2009},
which becomes important with decreasing structural size and increasing
surface area to volume ratio~\cite{cuenotPRB2004}. Additionally,
nanoscale surface stresses~\cite{cammarataPSS1994}, which also arise
from the coordination number reduction of surface
atoms~\cite{sunJPCM2002}, cause deformation of not only the surfaces,
but also the underlying bulk~\cite{liangPRB2005}, and can result in
unique physical properties such as phase
transformations~\cite{diaoNM2003}, or shape memory and
pseudoelasticity effects in FCC nanowires that are not observed in the
corresponding bulk material~\cite{parkPRL2005,liangNL2005}.

To study surface-dominated nanostructures, Park \emph{et al.} recently
developed the surface Cauchy-Born (SCB) model \cite{parkIJNME2006,
  parkPRB2007, parkCMAME2008}.  The idea is to seek an energy
functional of the form
\begin{displaymath}
  \Escb(y) = \int_\Om W(\D y) \dx + \int_{\pp\Om} \gamma(\D y, \nu)
  \ds,
\end{displaymath}
where $\Om \subset \R^3$ is an elastic body, $y : \Om \to \R^3$ a
deformation field, $W$ the bulk stored energy function, and
$\gamma$ a surface stored energy function. The potentials $W, \gamma$
are chosen such that $W(\mF)$ denotes the energy per unit volume in an
infinite crystal under the deformation $y(x) = \mF x$, while
$\gamma(\mF, \nu)$ is the surface energy per unit area of a half-space
with surface normal $\nu$, under the deformation $y(x) = \mF x$. Thus,
$W$ and $\gamma$ are {\em derived} from the underlying atomistic
model. For $W$ this is a well-understood idea \cite{BLBL:arma2002,
  E:2007a}; the novel approach in the SCB method is to apply the same
principle to the surface energy potential.

In contrast to the SCB method, most computational models (see, e.g.,
\cite{yvonnetCM2008, gaoNANO2006, heCM2009}) are based upon a finite
element discretization of the governing surface elasticity equations
of Gurtin and Murdoch~\cite{gurtinARMA1975}, where the constitutive
relation for the surface is linearly elastic or uses standard
hyperelastic strain energy functions \cite{javiliCMAME2009}.

The SCB model was successfully applied to various nanomechanical
boundary value problems, including thermomechanical
coupling~\cite{yunCMAME2008a}, resonant frequencies, and elucidating
the importance of nonlinear, finite deformation kinematics on the
resonant frequencies of both FCC metal~\cite{parkJMPS2008} and silicon
nanowires~\cite{parkJAP2008a,parkNANO2009}, bending of FCC
metal~\cite{yunPRB2009} nanowires, and electromechanical coupling in
surface-dominated nanostructures~\cite{parkCMAME2011}.

\begin{figure}[t]
  \includegraphics[height=5cm]{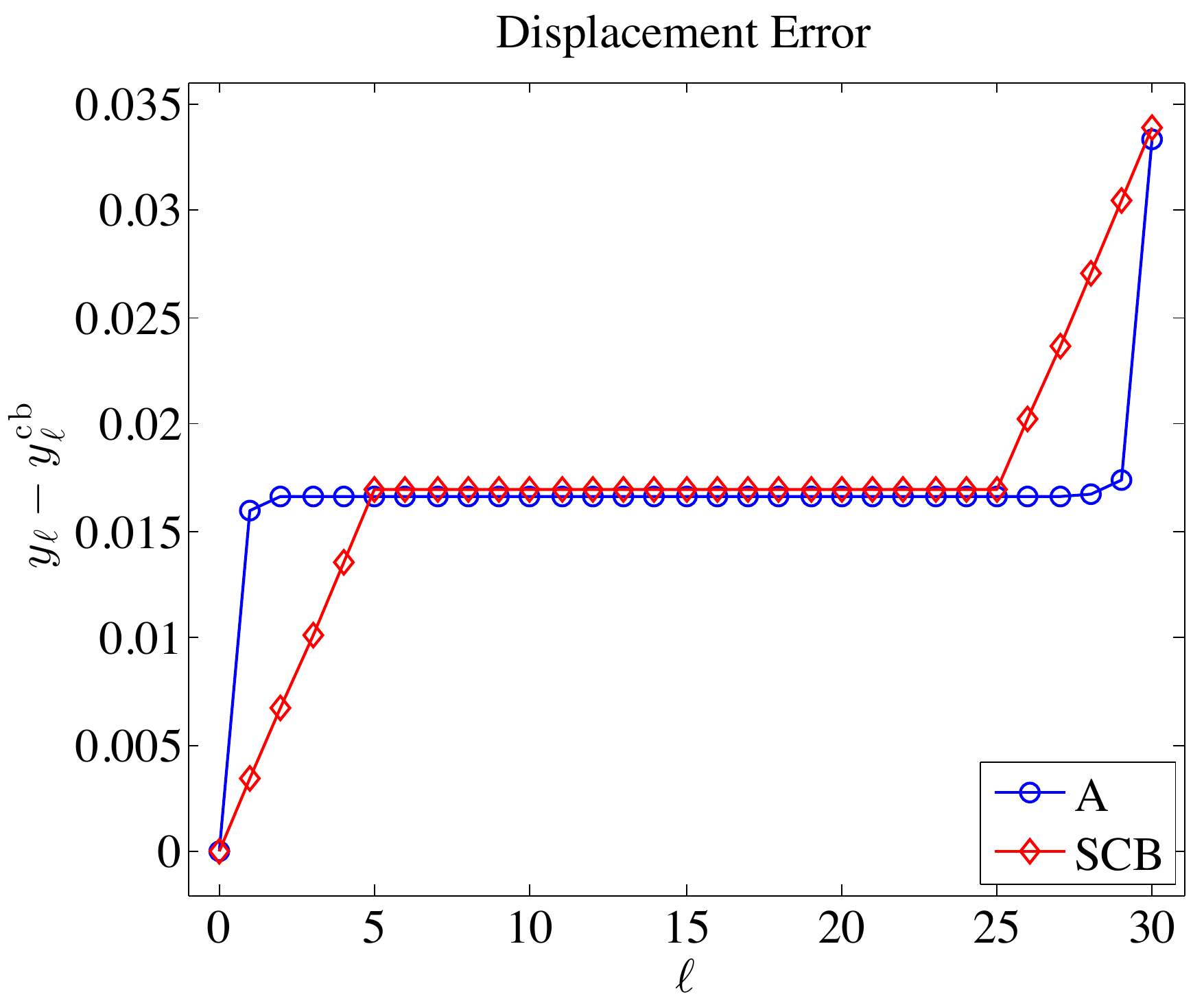}
  \qquad 
  \includegraphics[height=5cm]{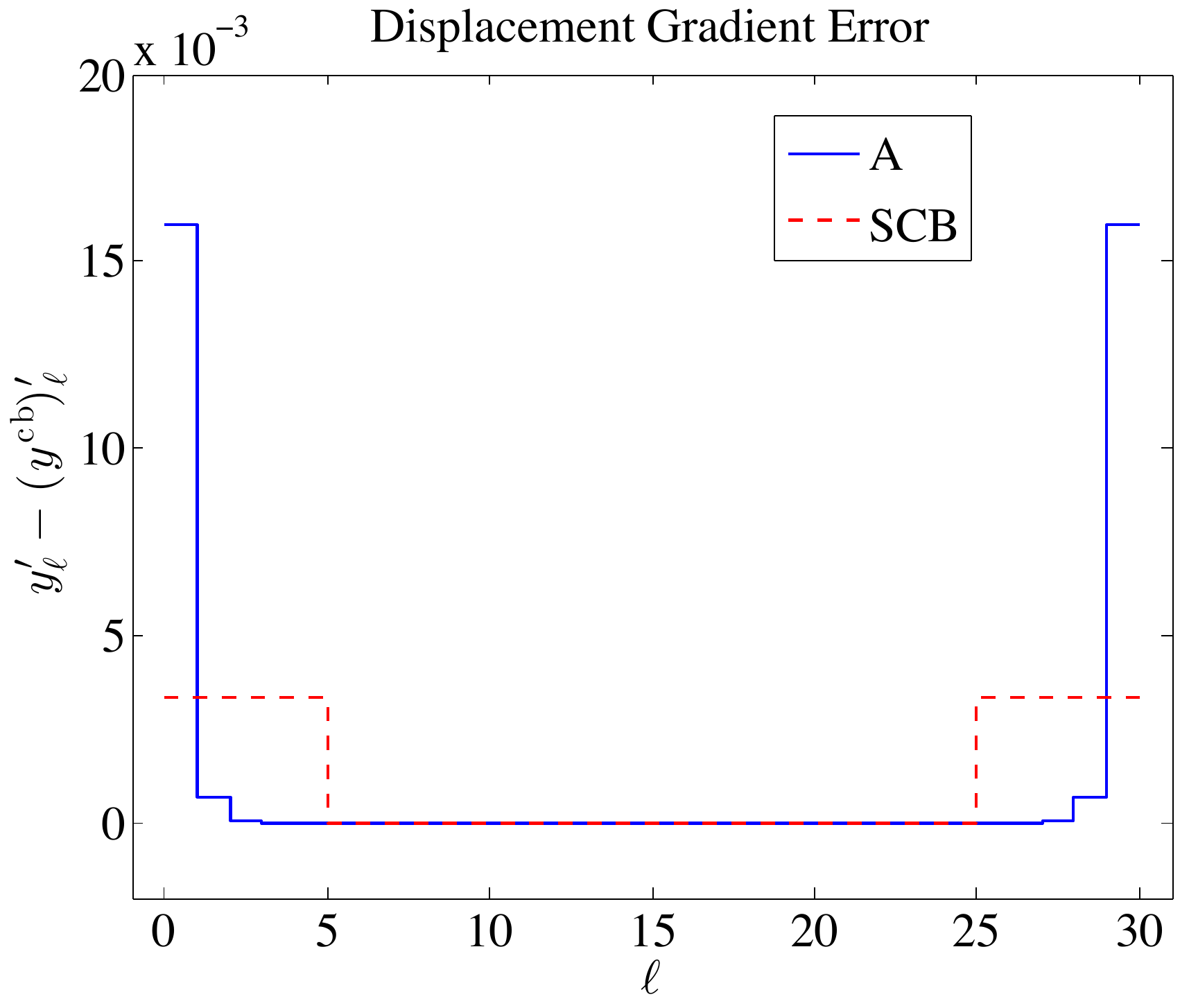}
  \caption{\label{fig:err_intro} Displacements and displacement
    gradients of an atomistic solution and a surface Cauchy--Born
    solution, relative to the bulk Cauchy--Born solution, for a 1D
    model problem. We observe unexpectedly high accuracy at the finite
    element nodes despite a large error in the displacement gradient.}
\end{figure}

{ The purpose of the present work is to initiate a mathematical
  analysis of the accuracy of the SCB method. We focus on the simplest
  setting where the only effect is a surface relaxation in normal
  direction.  While the SCB model does include surface physics that
  are neglected in the standard Cauchy--Born (CB) model, due to
  employing a coarse finite element discretisation it does not resolve
  the resulting boundary layer; see the numerical results in
  \cite{farsadIJNME2010} as well as see Figure \ref{fig:err_intro} for
  a 1D toy model demonstrating this. It is therefore {\it a priori}
  unclear to what extent the SCB improves upon the CB model. Figure
  \ref{fig:err_intro} suggests that, while the error in the
  displacement and displacement gradient is indeed of order $\OO(1)$
  in the boundary layer, the displacement error in finite element
  nodes is visually negligable, which would imply that the SCB model
  approximates the {\em mean strain} (and possibly other averaged
  quantities) to a much higher degree of accuracy. This was indeed
  observed in extensive numerical tests presented in
  \cite{parkPRB2007, parkCMAME2008, farsadIJNME2010}.  }

There is no traditional discretisation or approximation parameter
available with respect to which we might try to explain this
effect. Instead, our analysis measures the SCB error in terms of the
stiffness of the interaction potential. This enables us to identify a
suitable asymptotic limit for our analysis on a linearized model
problem. We confirm the analytical predictions with numerical
experiments on the fully nonlinear problem in 1D and and a periodic
semi-infinite 2D domain.

To the best of our knowledge, our work presents the first
approximation error results for the SCB method. Although our analysis
is elementary, it makes two important novel contributions: 1. We show
that the ``correct'' approximation parameter is the stiffness of the
interaction potential (however, Theil~\cite{Theil:2011} uses similar
ideas for an analysis of surface relaxation); and 2. We show that the
mean strain (which is an important quantity of interest) has a much
lower error than the strain field. {3. Our results show how to
  substantially improve the accuracy of the SCB method at little
  additional computational cost. Finally, we hope that this work will
  stimulate further research on computationally efficient multiscale
  methods for surface-dominated nanostructures.}

The issues we address here are closely related to the classical problem of numerical
methods for resolving boundary layers \cite{RoTyTo:2008}. The main
difference in our case is the discrete setting which does not give us
the opportunity to let the mesh-size tend to zero. For a mathematical
analysis of thin atomistic structures, surface energies and surface
relaxation we refer to \cite{Schmidt:2008, BrCi:2007, Theil:2011,
  ScSchZa:2011} and references therein. Our work also draws
inspiration from \cite{Dobson:2008c, qce.stab} where a similar
linearised model problem is used to analyze the accuracy of
atomistic-to-continuum coupling methods.

\section{Analysis of a 1D Model Problem}
\subsection{Atomistic model}
We consider a semi-infinite chain of atoms with reference positions
$\x \in \N$, and deformed positions $y_\x$, $\x \in \N$.  We assume
that the chain interacts through second-neighbour Morse pair
interaction. Hence, a deformed configuration $y$ has energy
\begin{equation}
  \label{eq:defn_Ea}
  \Ea(y) := \sum_{\x = 0}^\infty \B[ \phi(y_{\x+1} - y_\x) + 
  \phi(y_{\x+2} - y_\x) \B],
\end{equation}
where $\phi$ is a shifted Morse potential with stiffness parameter
$\alpha > 0$ and potential minimum $r_0 > 0$,
\begin{displaymath}
  \phi(r) = \exp(-2 \alpha (r - r_0)) - 2 \exp( -\alpha (r - r_0)) - \phi_0,
\end{displaymath}
where $\phi_0$ is chosen to that $W(1) = 0$, where $W(r) := \phi(r) +
\phi(2r) = 0$, $r_0$ is defined such that $W'(1) = 0$,
\begin{equation}
  \label{eq:defn_r0}
  r_0 = 1 + \frac{1}{\alpha} \log\B( \frac{1+2e^{-\alpha}}{1 + 2e^{-2\alpha}}\B),
\end{equation}
and $\alpha \geq 1 + \sqrt{3}$ remains a free parameter. This
restriction on $\alpha$ ensures that $\phi''(2) \leq 0$, which will be
convenient in the analysis. The shift of the potential by $\phi_0$
ensures that $\Ea$ is well-defined.

The potential $W$ is called the {\em Cauchy--Born stored energy
  density}. We have chosen the parameters in the Morse potential so
that $1$ is the minimizer of $W$, that is, we are working in
non-dimensional atomic units.

Since $\Ea$ is translation invariant, it is convenient to fix $y_0 =
0$. In that case, $y_\x$ is completely determined by the {\em forward
  differences} $y_\x' := y_{\x+1} - y_\x$. Hence we change coordinates
from the deformation $y_\x$ to the displacement gradient $u_\x :=
y_\x' - 1$, and rewrite $\Ea$ as
\begin{displaymath}
  \Ea(u) := \sum_{\x = 0}^\infty \B[ \phi(1+u_\x) +
  \phi(2+u_\x+u_{\x+1}) \B].
\end{displaymath}

The proof of the next result, which establishes that $\Ea$ is
well-defined, is given in the appendix.

\begin{proposition}
  \label{th:properties_Ea}
  $\Ea$ is well-defined and twice Fr\'echet differentiable in
  $\ell^1(\N)$ with first and second variations given by
  \begin{align*}
    \<\del\Ea(u), v \> =~& \sum_{\x = 0}^\infty \B[ \phi'(1+u_\x) v_\x
    + \phi'(2+u_\x+u_{\x+1})(v_\x+v_{\x+1}) \B], \\
    \< \ddel\Ea(u) v, w \> =~& \sum_{\x = 0}^\infty \B[ \phi''(1+u_\x)
    v_\x w_\x
    + \phi''(2+u_\x+u_{\x+1})(v_\x+v_{\x+1})(w_\x+w_{\x+1}) \B]. \\
  \end{align*}
\end{proposition}

\subsection{The Cauchy--Born and surface Cauchy--Born models}
The Cauchy--Born approximation is designed to model elastic bulk
behaviour in crystals. The stored energy density is chosen so that the
Cauchy--Born energy is exact under homogeneous deformations in the
absence of defects (such as surfaces). For the 1D model
\eqref{eq:defn_Ea} this yields
\begin{equation}
  \label{eq:defn_cb}
  \Ecb(y) := \int_0^\infty W(y') \dx, \qquad \text{for } y \in
  W^{1,1}_0(0, \infty),
\end{equation}
or equivalently, written in terms of the displacement gradient $u = y'
- 1$,
\begin{displaymath}
  \Ecb(y) = \int_0^\infty W(1+u) \dx, \qquad \text{for } u \in L^1(0, \infty),
\end{displaymath}
where $W(r) = \phi(r) + \phi(2r)$ was already defined above.

We consider a $\PI$ finite element discretisation of the Cauchy--Born
model. Let $X_h := \{X_0, X_1, \dots \} \subset \N$ be a strictly
increasing sequence of grid points with $X_0 = 0$, and let $h_j :=
X_{j+1} - X_j$. A $\PI$ discretisation of $y$ corresponds to a $\PO$
discretisation of the displacement gradient $u$, hence we define for
$(U_j)_{j = 0}^\infty \subset \R$, where $U_j$ denotes the
displacement gradient in the element $(X_j, X_{j+1})$,
\begin{displaymath}
  \Ecb_h(U) := \sum_{j = 0}^\infty h_j W(1+U_j).
\end{displaymath}

The Cauchy--Born approximation commits an error at the crystal
surface, which the surface Cauchy--Born (SCB) approximation aims to
rectify. The idea of the SCB method (in our 1D setting) is to define
\begin{equation}
  \label{eq:defn_scb}
  \Escb(y) := \int_0^\infty W(y') \dx + \gamma(y'(0)),
\end{equation}
and choose $\gamma$ such that the energy is exact under homogeneous
deformations, which yields the formula
\begin{equation}
  \label{eq:1d_gamma}
  \gamma(F) := - \smfrac12 \phi(2 F);
\end{equation}
see also Figure \ref{fig:1d_bdry}.  Converting to the displacement
gradient coordinate discretised by the $\PO$ finite element method we
obtain
\begin{displaymath}
  \Escb_h(U) := \Ecb_h(U) + \gamma(1+U_0).
\end{displaymath}

\begin{figure}[t]
  \includegraphics[width=6cm]{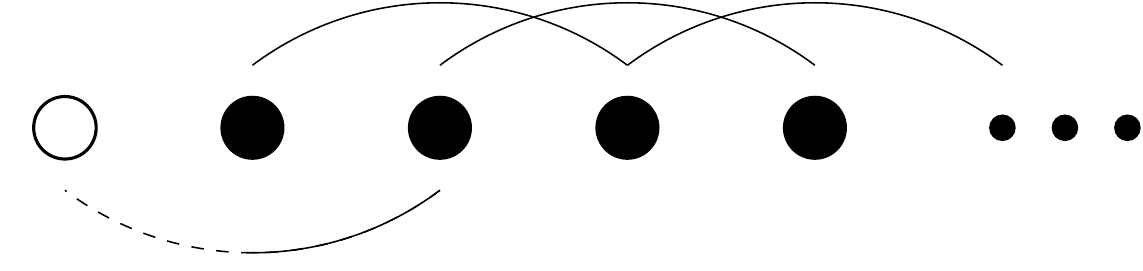}
  \caption{\label{fig:1d_bdry} Visualisation of \eqref{eq:1d_gamma}:
    the bond at the bottom of the graph is counted half in the
    Cauchy--Born model, even though it does not exist in the atomistic
    model, hence it gives a contribution $-\frac12 \phi(2y'(0))$ to
    the surface energy.}
\end{figure}

\begin{proposition}
  \label{th:properties_Ecb}
  $\Ecb_h$ and hence $\Escb_h$ are well-defined and twice Fr\'echet
  differentiable in the weighted space $\ell^1_h(X_h) := \{ V =
  (V_j)_{j = 0}^\infty \}$ equipped with the norm
  \begin{displaymath}
    \| V \|_{\ell^1_h} := \sum_{j = 0}^\infty h_j |V_j|.
  \end{displaymath}
  The first and second variations of $\Escb_h$ are given by
  \begin{align*}
    \< \del\Escb_h(U), V \> =~& \sum_{j = 0}^\infty h_j W'(1+U_j) V_j +
    \gamma'(1+U_0) V_0, \\
    \< \ddel\Escb_h(U) V, W \> =~& \sum_{j = 0}^\infty h_j W''(1+U_j)
    V_j W_j + \gamma''(1+U_0) V_0 W_0.
  \end{align*}
\end{proposition}

\subsection{Analysis of the linearized models}
The parameter $r_0$ for the Morse potential was chosen so that $1$ is
the minimizer of the Cauchy--Born stored energy function, which
implies that
\begin{equation}
  \label{eq:cbsoln}
  U^\cb_j := 0, \quad \text{ for } j = 0, 1, \dots
\end{equation}
is the ground state of $\Ecb_h$. More generally, $u^\cb := (0)_{\x =
  0}^\infty$ gives the {\em bulk ground state} of the crystal
described by the model \eqref{eq:defn_Ea}. We now consider
linearisations of $\Escb_h$ and $\Ea$ about the Cauchy--Born state:
$\del E(0) + \ddel E(0) u = 0$, where $E \in \{ \Ea, \Escb_h\}$.

From Proposition \ref{th:properties_Ecb} we obtain the linearised
optimality condition for $\Escb_h$,
\begin{align*}
  \gamma'(1) + (h_0 W''(1) + \gamma''(1)) U_0 =~& 0, \quad \text{and} \\
  h_j W''(1) U_j =~& 0 \qquad \text{for } j = 1, 2, \dots,
\end{align*}
which gives the linearised surface Cauchy--Born solution
\begin{equation}
  \label{eq:scbsoln}
  U^\scb_0 = \frac{-\gamma'(1)}{h_0 W''(1)+\gamma''(1)},
  \quad \text{and} \quad U^\scb_j = 0, \quad \text{for } j = 1, 2, \dots\,.
\end{equation}

From Proposition \ref{th:properties_Ea} we obtain the linearised
optimality condition for the atomistic model $\Ea$,
\begin{align*}
  \phi'(1) + \phi''(1) u_0 + \phi'(2) + \phi''(2) (u_0+u_1) =~& 0, \\
  \phi'(1) + \phi''(1) u_j + 2 \phi'(2) + \phi''(2)(u_{j-1}+2u_j +
  u_{j+1}) =~& 0, \quad j \geq 1,
\end{align*}
which, using the fact that $\phi'(1)+2\phi'(2) = W'(1) = 0$ can be
rewritten in the form
\begin{align*}
  [\phi''(1)+\phi''(2)] u_0 + \phi''(2) u_1 =~& \phi'(2), \\
  \phi''(2) u_{\x-1} + [\phi''(1) +2 \phi''(2)] u_\x + \phi''(2) u_{\x+1}
  =~& 0, \quad \x \geq 1.
\end{align*}
This finite difference equation can be easily solved explicitly, which
yields the solution
\begin{equation}
  \label{eq:atm_soln}
  u^\a_\x := \frac{\phi'(2) \lambda^{\x}}{\phi''(1) +
    \phi''(2)(1+\lambda)}, 
  \qquad \text{where} \quad
  \lambda = \frac{\sqrt{1 + 4\smfrac{\phi''(2)}{\phi''(1)}} - 1 - 2
    \smfrac{\phi''(2)}{\phi''(1)}}{2 \smfrac{\phi''(2)}{\phi''(1)}}
\end{equation}
is the unique solution in $(0, 1)$ of the characteristic equation
\begin{displaymath}
  \phi''(2) \lambda^2 + [\phi''(1)+2\phi''(2)]\lambda + \phi''(2) = 0.
\end{displaymath}

Since the expressions for \eqref{eq:scbsoln} and \eqref{eq:atm_soln}
are somewhat bulky we expand them in the stiffness parameter
$\alpha$. The rationale for expanding in this parameter is that all
models should coincide in the limit $\alpha \to \infty$. We hope,
however, that our results will also yield useful predictions for
moderate $\alpha$. The elementary proof is postponed to the appendix.

\begin{proposition}
  \label{th:asymp}
  Asymptotically as $\alpha \to 0$ we have the expansions
  \begin{align}
    \label{eq:asymp:scb}
    U^\scb_0 =~& \smfrac{e^{-\alpha}}{h_0 \alpha} \b[ 1 -
      \b(1 + \smfrac{2}{h_0} \b) e^{-\alpha} + \OO(e^{-2\alpha}) \b], \quad \text{and} \\ 
    \label{eq:asymp:atm}
    u^\a_0 =~& \smfrac{e^{-\alpha}}{\alpha} \b[1 - 4
    e^{-\alpha} + \OO(e^{-2\alpha}) \b].
  \end{align}
\end{proposition}

\begin{remark}
  The asymptotic expansions \eqref{eq:asymp:scb} and
  \eqref{eq:asymp:atm} justify {\it a posteriori} the linearisation
  since they show that the displacements from the Cauchy--Born state
  are indeed small in the limit as $\alpha \to \infty$.
\end{remark}

\subsection{Error estimates}
We first note that each $\PO$ function $U = (U_j)_{j = 0}^\infty$ can
be understood as a lattice function $u = (u_\x)_{\x = 0}^\infty$
through the interpolation
\begin{displaymath}
  u_\x = U_j \quad \text{for } \x = X_j, \dots, X_{j+1}-1, \quad j \in \N.
\end{displaymath}
With this interpolation we obtain $u^\cb = 0$ and $u^\scb$ from the
linearized CB and SCB solutions $U^\cb$ and $U^\scb$, given in
\eqref{eq:scbsoln}.

We are interested in the improvement the SCB model gives over the pure
Cauchy--Born model, that is, we wish to measure the relative errors
\begin{displaymath}
  \Err_p := \frac{ \|u^\scb - u^\a \|_{\ell^p}}{\|u^\cb - u^\a
      \|_{\ell^p}}
    = \frac{ \|u^\scb - u^\a \|_{\ell^p}}{\| u^\a \|_{\ell^p}}.
\end{displaymath}
Of particular interest are the uniform error $\Err_\infty$ and the
error in the energy-norm $\Err_2$. We shall consider two separate
cases: $h_0 > 1$ and $h_0 = 1$.

\begin{proposition}[Strain error]
  \label{th:errp}
  Let $p \in [1, \infty]$ and $h_0 > 1$, then 
  \begin{equation}
    \label{eq:errp}
    \Err_p = C_p + \OO(e^{-\alpha}),
 \end{equation}
 where $\smfrac12 \leq C_p \leq 2$.  If $h_0 = 1$, then
 \begin{equation}
   \label{eq:errp_1}
   \Err_p = 2^{1/p} e^{-\alpha} + \OO(e^{-2\alpha}).
 \end{equation}
\end{proposition}
\begin{proof}
  We consider the case $h_0 = 1$ first. In that case
  \eqref{eq:asymp:atm} gives us
  \begin{align*}
    \bg(\sum_{\x = 1}^\infty |u_\x^\scb - u_\x^\a|^p\bg)^{1/p} =~&
    \bg(\sum_{\x = 1}^\infty |u_\x^\a|^p\bg)^{1/p} 
    = \lambda u_0^\a \b(1 - \lambda^p\b)^{-1/p},
  \end{align*}
  and similarly, $\| u^\a \|_{\ell^p} = u_0^\a (1-\lambda^p)^{-1/p}$.
  Using the asymptotic expansions \eqref{eq:asymp_lambda} for
  $\lambda$ it is straightforward to show that 
  \begin{displaymath}
    (1-\lambda^p)^{-1/p} = 1 + \OO(\lambda) =
    1 + \OO(e^{-\alpha});
  \end{displaymath}
  hence employing also \eqref{eq:asymp:atm} we obtain
  \begin{equation}
    \label{eq:errp:10}
    \| u^\a \|_{\ell^p} = \frac{e^{-\alpha}}{\alpha} +
    \OO\b(\smfrac{e^{-2\alpha}}{\alpha}\b), \quad \text{and} \quad
    \bg(\sum_{\x = 1}^\infty |u_\x^\scb - u_\x^\a|^p\bg)^{1/p}
    = \frac{e^{-2\alpha}}{\alpha} + \OO\b(\smfrac{e^{-3\alpha}}{\alpha}\b).
  \end{equation}

  For $\x = 0$, since $h_0 = 1$, we have
  \begin{align*}
    \b| u_0^\scb - u_0^\a \b| = \B| \smfrac{e^{-\alpha}}{\alpha} \b[ 1
    - 3 e^{-\alpha} + \OO(e^{-2\alpha}) \b] -
    \smfrac{e^{-\alpha}}{\alpha} \b[1 - 4 e^{-\alpha} +
    \OO(e^{-2\alpha})\b] 
    = \smfrac{e^{-2\alpha}}{\alpha} + \OO\b( \smfrac{e^{-3\alpha}}{\alpha} \b).
  \end{align*}
  Combined with \eqref{eq:errp:10} this gives
  \begin{align*}
    \Err_p = \frac{\| u^\a - u^\scb \|_{\ell^p}}{\|u^\a\|_{\ell^p}} 
    = \frac{2^{1/p} \smfrac{e^{-2\alpha}}{\alpha} +
      \OO\b(\smfrac{e^{-3\alpha}}{\alpha}\b)}{\frac{e^{-\alpha}}{\alpha} +
      \OO\b(\smfrac{e^{-2\alpha}}{\alpha}\b)} 
    = 2^{1/p} e^{-\alpha} + \OO(e^{-2\alpha}),
  \end{align*}
  which concludes the proof of \eqref{eq:errp_1}.

  In the case $h_0 > 1$ the convenient cancellation of first-order
  terms in $u_0^\scb - u_0^\a$ does not occur. Instead, using
  \eqref{eq:errp:10} we obtain
  \begin{align*}
    \| u^\a - u^\scb \|_{\ell^p} =~& \frac{e^{-\alpha}}{\alpha}
    \B( \b| 1 - \smfrac{1}{h_0}\b|^p + \sum_{\x = 1}^{X_1-1}
    \b|\smfrac{1}{h_0}\b|^p \B)^{1/p} + \OO\b(\smfrac{e^{-2\alpha}}{\alpha}\b)
  \end{align*}
  This immediately gives \eqref{eq:errp}.
\end{proof}

We see from \eqref{eq:errp} that if we use a coarse finite element
mesh up to the boundary, then the error in the displacement gradient
will be typically of the order $50\%$ or more. By contrast, if we
refine the finite element mesh to atomistic precision at the boundary
then the relative error is exponentially small in the stiffness
parameter $\alpha$.

The quantity $\Err_p$ measures the error in a pointwise
sense. However, in some cases we are only interested in correctly
reproducing certain macroscopic quantities such as the mean strain
error
\begin{displaymath}
  \ol{\Err} := \B|\frac{\sum_{\x = 0}^\infty (u_\x^\scb -
    u_\x^\a)}{\sum_{\x = 0}^\infty u_\x^\a} \B|.
\end{displaymath}
Note that, up to higher order terms, this error also bounds the error
in the displacements at the finite element nodes, which we observed in
Figure \ref{fig:err_intro} to be much smaller than the strain error.

In the following result we confirm that, indeed, the mean strain error
is an order of magnitude smaller than the pointwise strain error.

\begin{proposition}[Mean strain error]
  \label{th:errm}
  Asymptotically as $\alpha \to \infty$, the mean strain error
  satisfies
  \begin{equation}
    \label{eq:errm}
    \ol{\Err} =  2 \b( 1 - \smfrac{1}{h_0}\b)
    e^{-\alpha} + \OO(e^{-2\alpha}).
  \end{equation}
\end{proposition}
\begin{proof}
  We first compute the mean strains in the atomistic and the SCB
  models. For the atomistic model we have
 \begin{displaymath}
    \ol{u}^\a := \sum_{\x = 0}^\infty u_\x^\a = \frac{u_0^\a}{1 -
      \lambda}
  \end{displaymath}
  Since $(1-\lambda)^{-1} = 1 + e^{-\alpha} + \OO(e^{-2\alpha})$ we
  obtain 
  \begin{displaymath}
    \ol{u}^\a = \smfrac{e^{-\alpha}}{\alpha}\b[(1 - 4 e^{-\alpha})(1 +
    e^{-\alpha}) + \OO(e^{-2\alpha})\b]
    = \smfrac{e^{-\alpha}}{\alpha}\b[1 - 3 e^{-\alpha} + \OO(e^{-2\alpha})\b],
  \end{displaymath}

  For the SCB model, we have
  \begin{displaymath}
    \ol{u}^\scb = \sum_{j = 0}^\infty h_j U_j^\scb =  h_0 U_0^\scb
    = \smfrac{e^{-\alpha}}{\alpha} \b[ 1 -
      \b(1 + \smfrac{2}{h_0} \b) e^{-\alpha} + \OO(e^{-2\alpha})\b],    
  \end{displaymath}
  and hence the error is given by
  \begin{displaymath}
    \ol{u}^\scb - \ol{u}^\a = 2 \b( 1 - \smfrac{1}{h_0}\b)
    \smfrac{e^{-2\alpha}}{\alpha} + \OO\b(\smfrac{e^{-3\alpha}}{\alpha}\b).
  \end{displaymath}
  This immediately implies \eqref{eq:errm}.
\end{proof}

\begin{remark}
  Since $\Ea$ and $\Escb$ are Fr\'{e}chet differentiable in suitable
  function spaces it should be possible, using nonlinear analysis
  techniques such as the inverse function theorem, to extend the
  results from the linearized model problem to the fully nonlinear
  problem, provided that the stiffness parameter $\alpha$ is
  sufficiently large. Techniques of this kind have been used, for
  example, in \cite{Theil:2011}.
\end{remark}

\subsection{Numerical results}
\label{sec:num1d}
We confirm through numerical experiments that the results of
Propositions \ref{th:errp} and \ref{th:errm} are still valid in the
nonlinear setting. In these experiments we choose $r_0 = 1$ instead of
\eqref{eq:defn_r0}, choose a finite chain with $31$ atoms, and let
$\alpha$ vary between $2$ and $7$. For experiments with $h_0 = 5$ the
gridpoints for the Cauchy--Born and SCB models are chosen as $X = (0,
5, 10, \dots, 30)$. For experiments with $h_0 = 1$, the gridpoints are
chosen as $X = (0, 1, 5, \dots, 25, 29, 30)$.

The results of the experiments are displayed in Figures
\ref{fig:err1d_err2} and \ref{fig:err1d_errm}. All results except for
the relative error in the mean strain with $h_0 = 1$ confirm our
analytical results in the linearized case. We have, at present, no
explanation why the mean strain error $\ol\Err$ with $h_0 = 1$ is of
the order $\OO(e^{-3\alpha})$ instead of $\OO(e^{-2\alpha})$. A finer
asymptotic analysis in the linearized case would in fact give the
expansion $\ol{\Err} = 2 e^{-2\alpha} + \OO(e^{-3\alpha})$.

\begin{figure}[t]
  \includegraphics[width=10cm]{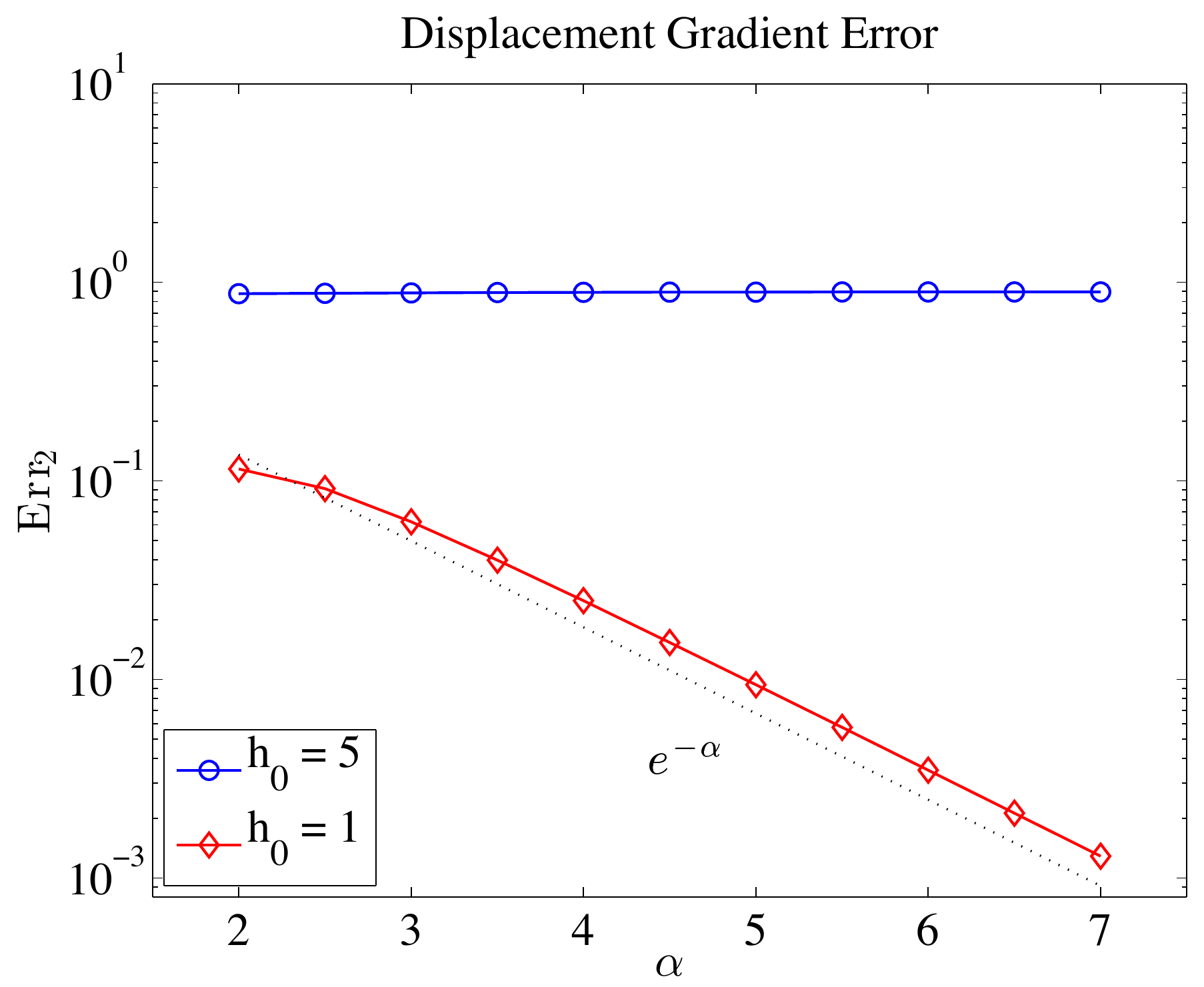}
  \caption{\label{fig:err1d_err2} Relative error in the
    $W^{1,2}$-seminorm of the 1D nonlinear SCB model for varying
    stiffness parameter $\alpha$ and two types of finite element
    grids; cf. Section \ref{sec:num1d}.}
\end{figure}

\begin{figure}[t]
 \quad
  \includegraphics[width=10cm]{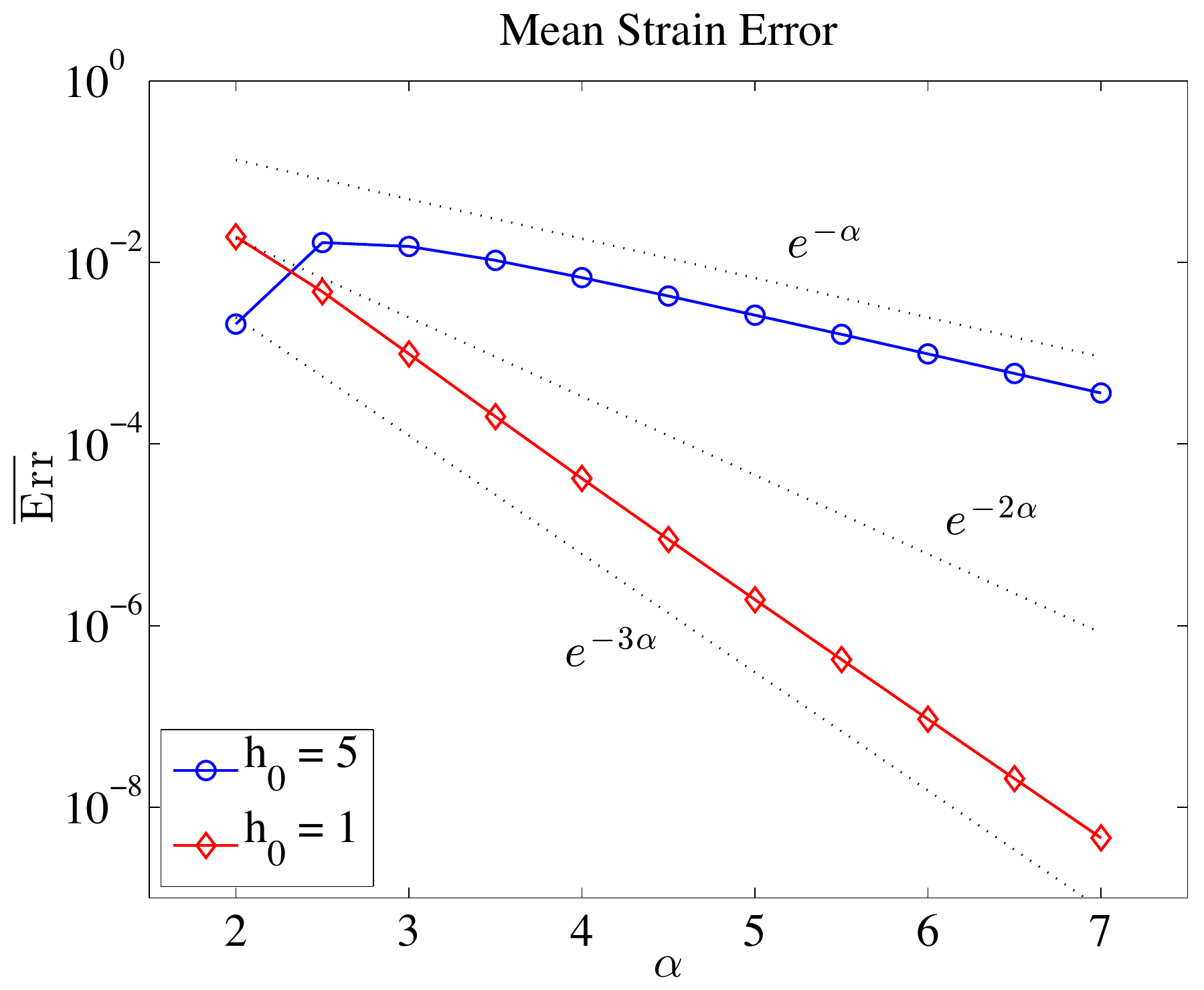}
  \caption{\label{fig:err1d_errm} Relative error in the mean strain of
    the 1D nonlinear SCB model for varying stiffness parameter
    $\alpha$ and two types of finite element grids; cf. Section
    \ref{sec:num1d}.}
\end{figure}

\section{Numerical Results in 2D}
\label{sec:2d}
In this section we investigate numerically, to what extent the 1D
results might extend to the 2D setting. We will formulate a
problem in a semi-infinite strip, where we expect relaxation only in
the normal direction to the surfaces. Hence the problem reduces to a
1D problem for the displacements in that direction. The 1D analysis
can be applied to this case with only minor changes, and we therefore
expect the same behaviour as in the 1D case. This is fully confirmed
by the results of our numerical experiment.

\subsection{Formulation of the SCB method}
In 2D one expects (this is rigorously proven only for large stiffness
parameter $\alpha$ \cite{Theil:2006}) that the ground-state under
Morse potential interaction is the triangular lattice. Hence we choose
as the atomistic reference configuration a subset $\L \subset \mA
\Z^2$, where
\begin{displaymath}
  \mA = \mymat{ 1 & 1/2 \\ 0 & \sqrt{3}/2 }.
\end{displaymath}
For future reference, we define $a_1 := (1,0), a_2 := (1/2,
\sqrt{3}/2)$ and $a_3 := (-1/2, \sqrt{3}/2)$, which are the directions
of nearest-neighbour bonds. 

Specifically, we choose $N_1, N_2 \in \N$ and define
\begin{displaymath}
  \L := \b\{ \mA (n_1, n_2)^T \in \Z^2 \bsep 1 < n_1 \leq N_1, 0 \leq
  n_2 \leq N_2 \b\},
\end{displaymath}
as the periodic cell of the semi-infinite strip $\L^\per := \{ \mA
(n_1, n_2)^T \in \Z^2 \bsep 0 \leq n_2 \leq N_2 \}$; cf. Figure
\ref{fig:scb_strip}. The corresponding continuous domain is $\Om :=
\mA ( (0, N_1] \times (0, N_1] )$.

\begin{figure}[t]
  \includegraphics[height=4.2cm]{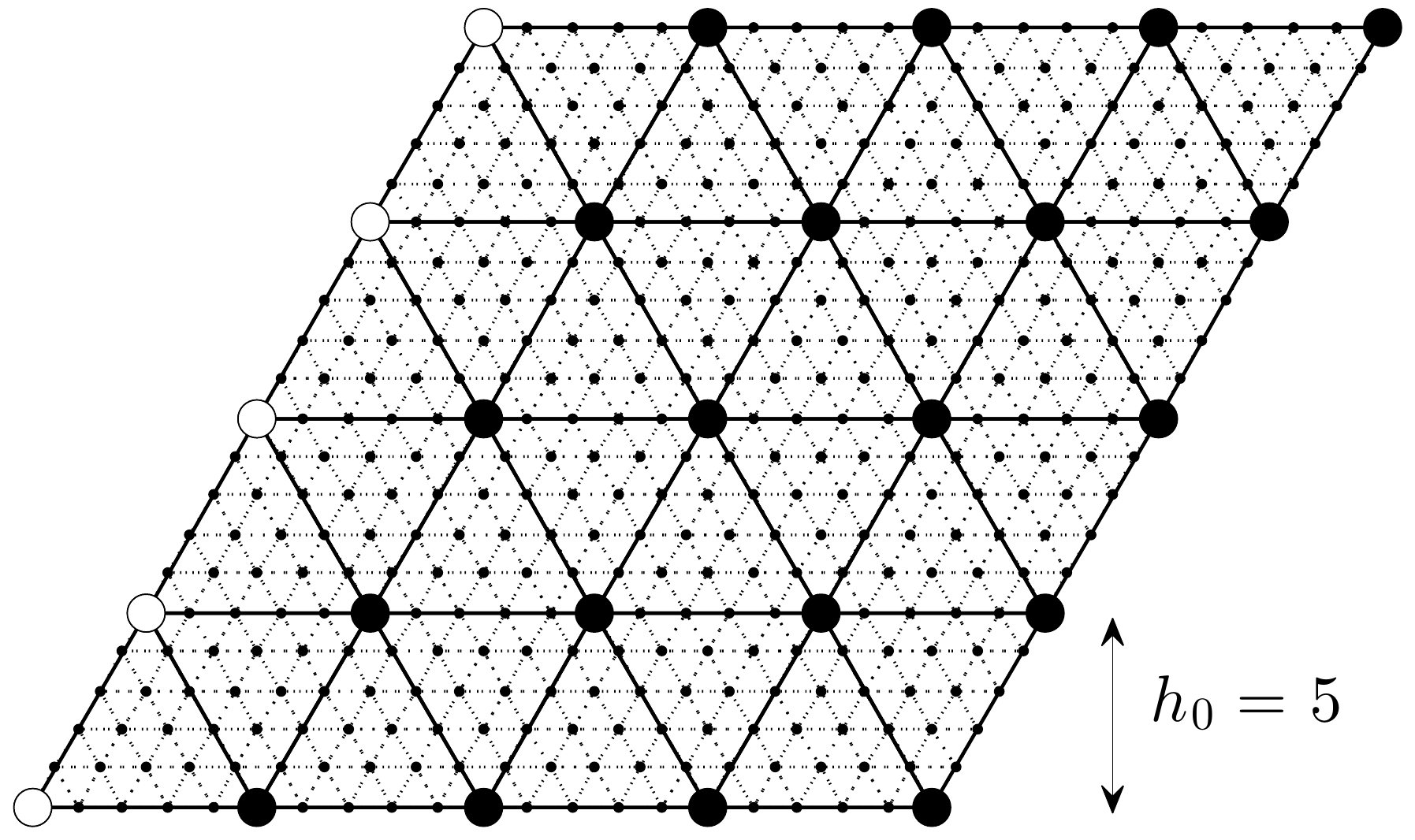}  \hspace{-4mm}
  \includegraphics[height=4.2cm]{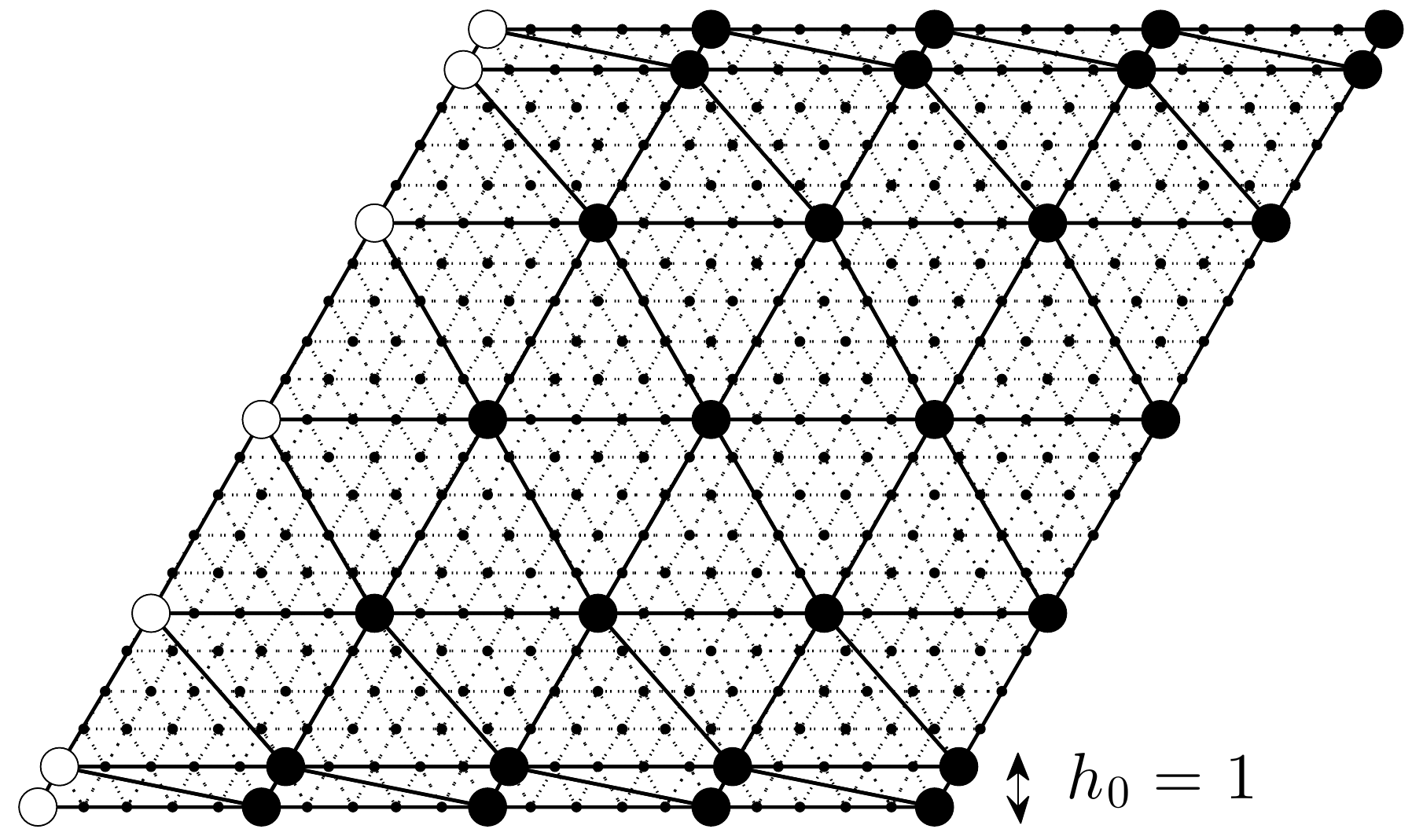}\\
  (a) \hspace{6cm} (b) 
  \caption{\label{fig:scb_strip} Computational domain used in the
    numerical experiment described in Section \ref{sec:2d}. The small
    disks denote the set $\L$; the dotted grid is the
    micro-triangulation $\Tm$; the the large black disks denote the
    finite element nodes; the large white discs denote finite element
    nodes that are periodically repeated; the black lines denote the
    macro-triangulation $\Th$.}
\end{figure}

An admissible deformed configuration is a map $y : \L^\per \to \R^2$,
which is periodic in the $a_1$-direction, that is, $y(\xi+N_1 a_1) =
y(\xi) + N_1 a_1$.

For simplicity we consider only second-neighbour interactions
(measured in hopping distance). For each $\xi \in \L$ let $\Ng_\xi :=
\{ \eta \in \L^\per \sep |\eta - \xi| \leq 2 \}$ denote the
interaction neighbourhood of $\xi$, then the potential energy of a
deformed configuration is given by
\begin{displaymath}
  \Ea(y) := \sum_{\xi \in \L} \frac12 \sum_{\eta \in \Ng_\xi} \phi\b(
  |y(\eta) - y(\xi)|\b),
\end{displaymath}
where $\phi$ is again the Morse potential. 

To evaluate the deformation gradient $\D y$ of a discrete deformation
$y$, we note that $\L^\per$ has a natural triangulation $\Tm$ (see
Figure \ref{fig:scb_strip}), and identify $y$ with its continuous
piecewise affine interpolant in $\PI(\Tm; \R^2)$.

Let $\Th$ be a {\em coarse} triangulation of $\Omega$ (which can be
repeated periodically) and let $\PI(\Th;\R^2)$ denote the space of
continuous and piecewise affine deformations of $\Om$, such that
$y_h(x + N_1 a_1) = y_h(x) + N_1 a_1$, then the SCB energy of a
deformation $y_h \in \PI(\Th; \R^2)$ is given by
\begin{displaymath}
  \Escb(y_h) = \int_\Om W(\D y_h) \dx + \int_{\Gamma} \gamma(\D y_h,
  \nu) \dx,
\end{displaymath}
where $\Gamma \subset\pp\Omega$ denotes the free boundary, that is the
portion of the boundary with normal $\nu = \pm (0, 1)$, $W$ is the
Cauchy--Born stored energy function and $\gamma$ the SCB surface
energy function, which are defined as follows:
\begin{figure}[t]
  \includegraphics[height=3cm]{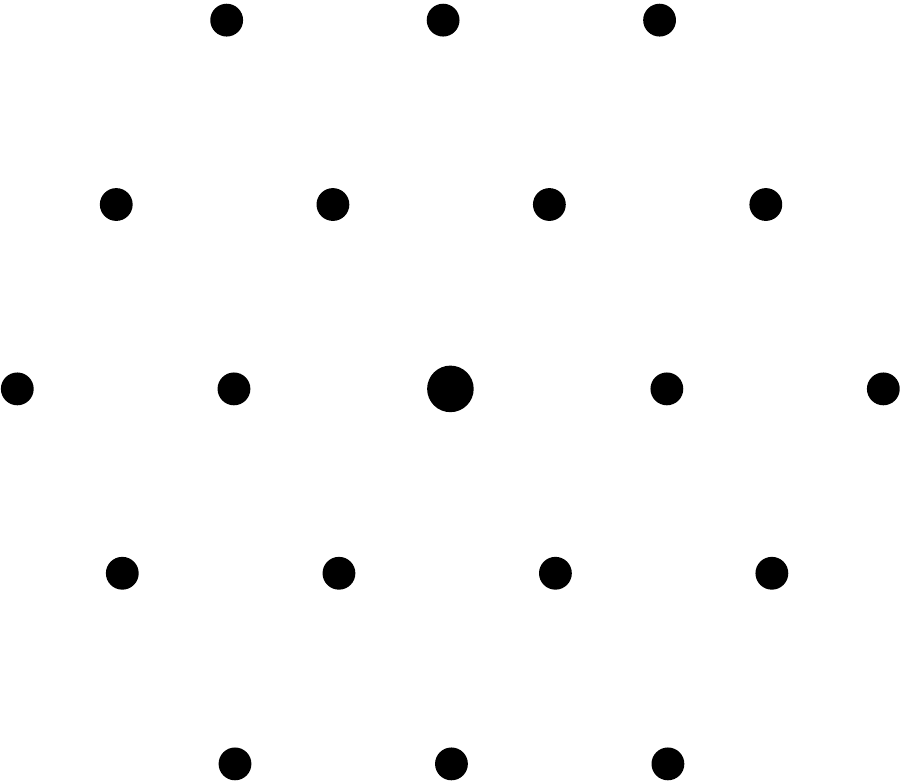} \quad \qquad
  \includegraphics[height=3cm]{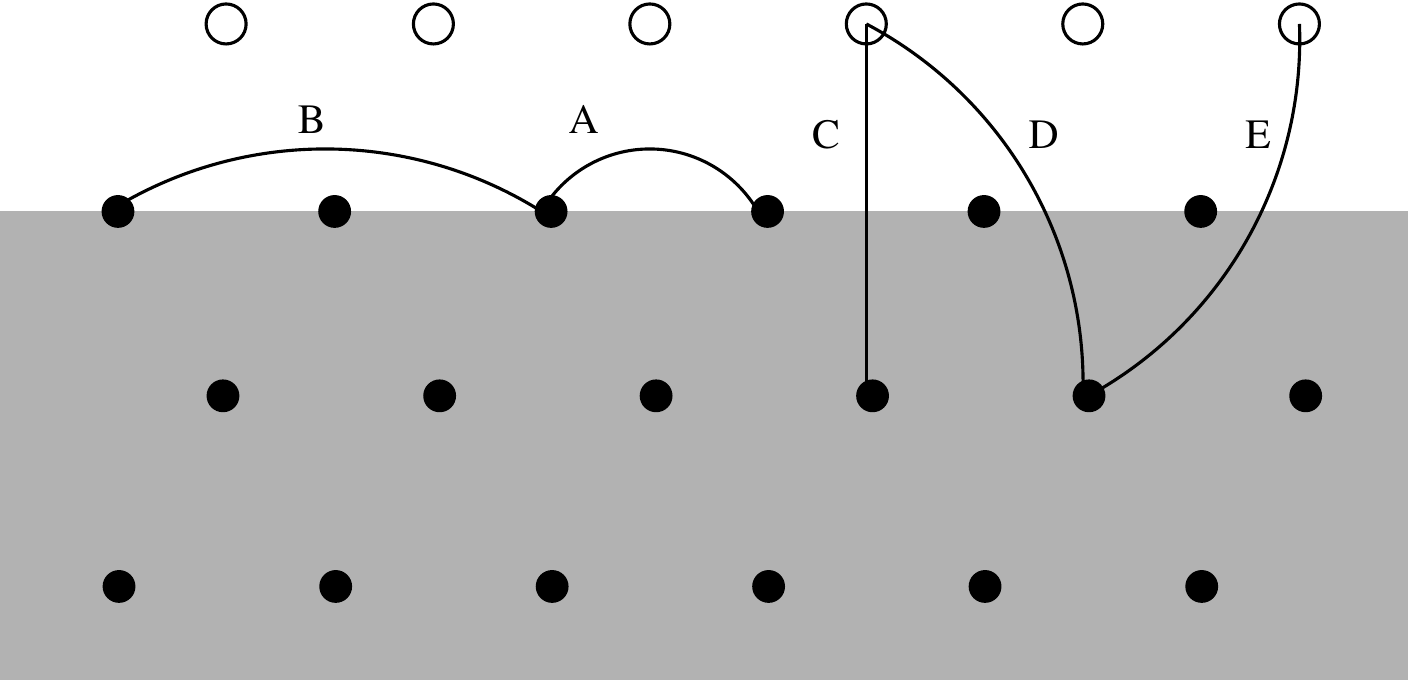} \\
 {\small (a) \hspace{6.5cm} (b) \hspace{1.5cm}}
  \caption{\label{fig:2d_flat} (a) Third interaction
    neighbourhood. (b) Construction of $\gamma$: Bonds A, B are
    underestimated by the Cauchy--Born approximation (counted only
    half), while the bonds C, D, E are overestimated (they do not
    exist in the atomistic model but are counted half in the
    Cauchy--Born model).}
\end{figure}
\begin{itemize}
\item If we denote by $\Ng_\cb$ the interaction neighbourhood of the
  origin in the infinite lattice $\mA \Z^2$ (see Figure
  \ref{fig:2d_flat}(a)), then the Cauchy--Born stored energy function
  is given by
  \begin{displaymath}
    W(\mF) = \frac{1}{\det\mA}\sum_{\eta \in \Ng_\cb} \phi\b(| \mF \eta|\b).
  \end{displaymath}
\item To define $\gamma$, we assume throughout that all surfaces of
  $\Om$ are aligned with one of the three directions $a_1, a_2$, or
  $a_3$, that is, $\nu \perp a_j =: \nu^\perp$. Then the requirement
  that the SCB energy is exact under homogeneous deformations, in
  domains without corners, yields the expression
  \begin{align*}
    \gamma(\mF, \gamma) =~& \smfrac12 \phi\b( |\mF \nu^\perp| \b) +
    \smfrac12 \phi\b(2 |\mF \nu^\perp|\b) \\
    & - \smfrac12 \phi\b( \sqrt{3} |\mF \nu| \b)
    - \smfrac12 \phi\b( 2 |\mF \mQ_{12} \nu| \b) 
    -\smfrac12 \phi\b( 2 |\mF \mQ_{12}^T \nu| \b),
  \end{align*}
  where $\mQ_{12}$ denotes a rotation through arclength $2\pi/12$; see
  Figure \ref{fig:2d_flat}(b) for an illustration. A rigorous proof of
  this formula follows immediately from Shapeev's bond density lemma
  \cite{Shapeev2011}.
\end{itemize}

\subsection{Numerical results}
\label{sec:2d:num}
In the numerical experiments we consider two types of finite element
grids: a uniform grid with spacing $h = h_0 = 5$ (cf. Figure
\ref{fig:scb_strip}(a)), and a grid with an additional layer of elements
at the free boundary, atomic spacing $h_0 = 1$ in the normal direction
and uniform spacing $h = 5$ in the tangential direction (cf. Figure
\ref{fig:scb_strip}(b)). We will again measure the following relative
errors:
\begin{displaymath}
  \Err_2 := \frac{\| \D y^\scb_h - \D y^\a \|_{L^2}}{\| \D y_h^\cb -
    \D y^\a \|_{L^2}}, \qquad \text{and} \qquad
  \ol{\Err} := \bg|\frac{ \int_\Om (\D y_h^\scb - \D
    y^\a)\dx}{\int_\Om (\D y_h^\cb - \D y^\a) \dx} \bg|,
\end{displaymath}
where $y^\a, y_h^\scb,$ and $y_h^\cb$ denote the minimizers of,
respectively, $\Ea, \Escb$, and $\Escb$ with $\gamma = 0$. That is,
$\Err_2$ and $\ol{\Err}$ measure the improvement of SCB over the pure
Cauchy--Born model.

The numerical results are displayed in Figures
\ref{fig:err_strip_err2} and \ref{fig:err_strip_errm}. Although the
numerical results do not as clearly display the predicted convergence
rates, they do seem to approach these rates for increasing values of
$\alpha$. What is again clear is that the average strain has a much
higher accuracy than the strain field, and that the additional mesh
layer also substantially improves the accuracy of the method. We also
note that we now observe essentially the predicted rate $e^{-2\alpha}$
for the mean-strain error of the enhanced SCB model, instead of the
unexpected rate $e^{-3\alpha}$.

\begin{figure}[t]
  \includegraphics[width=10cm]{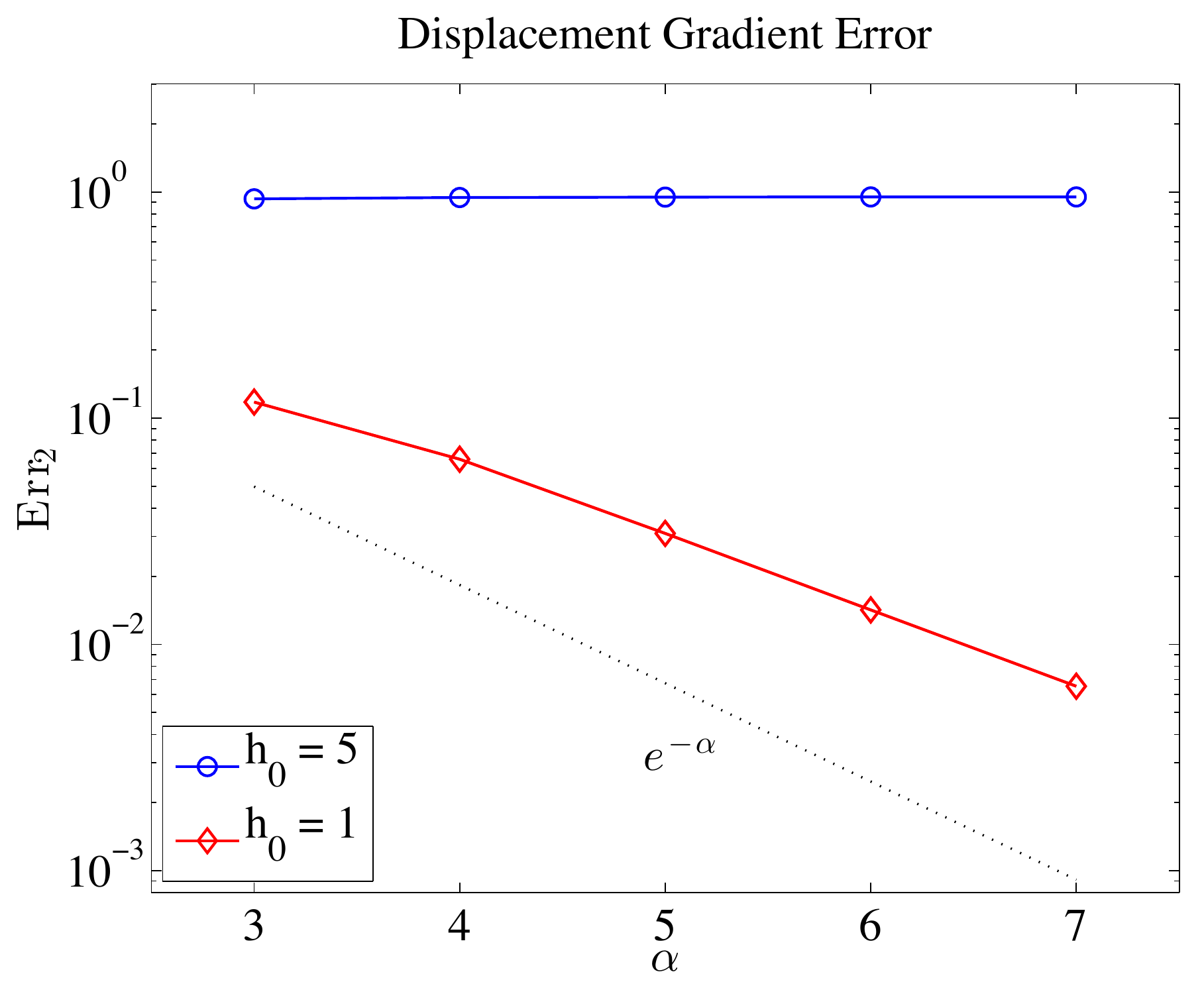}
  \caption{\label{fig:err_strip_err2} Relative error in the
    $W^{1,2}$-seminorm of the 2D SCB model in the flat interface
    example described in Section \ref{sec:2d}, for varying stiffness
    parameter $\alpha$ and two types of finite element grids.}

\end{figure}

\begin{figure}[t]
 \includegraphics[width=10cm]{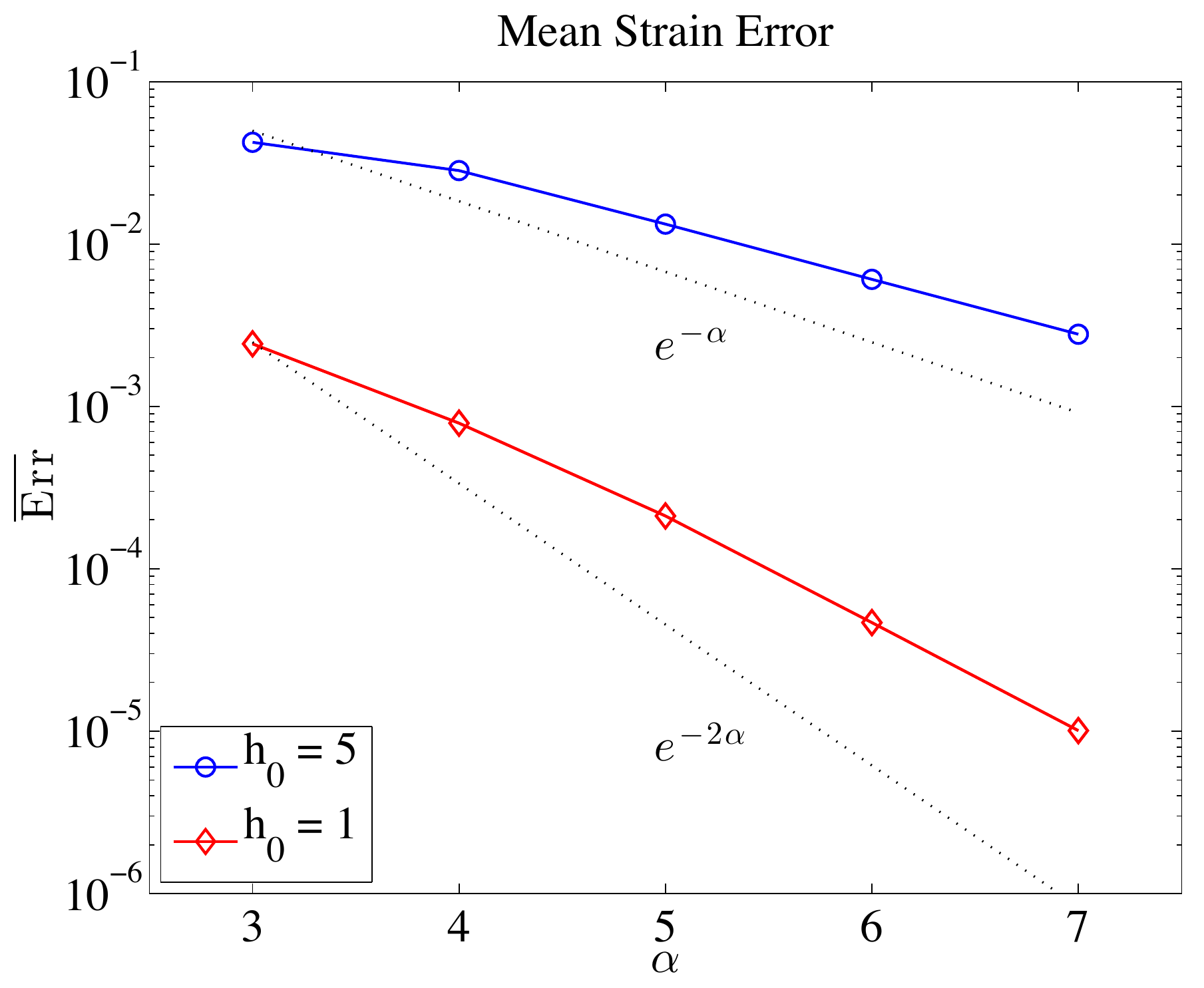}
 \caption{\label{fig:err_strip_errm} Relative error for the mean
   strain of the 2D SCB model applied to the flat interface example
   described in Section \ref{sec:2d}, for varying stiffness parameter
   $\alpha$ and two types of finite element grids.}
\end{figure}

\section*{Conclusion}
We presented an error analysis of the SCB method in the case where the
dominant effect is surface relaxation in the normal direction. Our
main results are: 1. We showed that the ``correct'' approximation
parameter is the stiffness of the interaction potential. 2. We showed
that the mean strain (which is an important quantity of interest) has
a much lower error than the strain field. 3. We showed that adding a
single mesh layer at the free boundary with atomic spacing in the normal
direction yields a substantial improvement to the accuracy of the SCB
method with minimal increase in the computational cost.

We also performed numerical experiments for domains with corners,
which remain inconclusive so far. At corners there is an interplay
between the normal stress and tangential stress of adjacent edges,
which creates additional elastic fields. A finer analysis of this case
is still required. In particular, it would be interesting to
understand whether normal or tangential forces dominate the bahaviour
of the system in that case.

\appendix

\section{Proofs}

\begin{proof}[Proof of Propositions \ref{th:properties_Ea} and \ref{th:properties_Ecb}]
  For each $\x \in \N$ we have
  \begin{align*}
    \phi(1+u_\x) + \phi(2+u_\x+u_{\x+1})
    =~& \phi(1) + \phi'(1) u_\x + \smfrac12 \phi''(\theta_\x^{(1)})
    |u_\x|^2 \\
     & + \phi(2)  + \phi'(2) (u_\x+u_{\x+1}) \smfrac12
     \phi''(\theta_\x^{(2)}) |u_\x+u_{\x+1}|^2.
  \end{align*}
  where $(\theta_\x^{(j)}-j) \in \ell^1$ by Taylor's theorem. Since
  $\phi(1) + \phi(2) = 0$, summing over $\x \in \N$ and noting that
  the first-order terms cancel, yields
  \begin{displaymath}
    \Ea(y) \leq C \| u \|_{\ell^2}^2.
  \end{displaymath}
  Note that this seemingly requires only that $u \in \ell^2$, however,
  the series converges absolutely only if $u \in \ell^1$.

  Repeating the argument for a perturbation from a general state
  $\Ea(u + v)$ shows the Fr\'echet differentiability of $\Ea$. 

  The same argument can be applied to prove Proposition
  \ref{th:properties_Ecb}.
\end{proof}

\begin{proof}[Proof of Proposition \ref{th:asymp}]
  Inserting the definition of $r_0$ from \eqref{eq:defn_r0} into
  $\phi''(1)$ yields
  \begin{align*}
    \phi''(1) =~& 4 \alpha^2 e^{-2\alpha (1 - r_0)} - 2 \alpha^2
    e^{-\alpha(1-r_0)} \\
    =~&  4 \alpha^2 \B( \frac{1 + 2 e^{-\alpha}}{1 + 2 e^{-2\alpha}}
    \B)^2 - 2 \alpha^2 \B( \frac{1 + 2 e^{-\alpha}}{1 + 2
      e^{-2\alpha}} \B).
  \end{align*}
  Expanding
  \begin{displaymath}
    \frac{1 + 2 e^{-\alpha}}{1 + 2 e^{-2\alpha}}  = 1 + 2 e^{-\alpha}
    + \OO(e^{-2\alpha}),
  \end{displaymath}
  we obtain
  \begin{align}
    \notag
    \phi''(1) =~& 4 \alpha^2 (1 + 4 e^{-\alpha}) - 2\alpha^2 (1 + 2
    e^{-\alpha}) + \OO(\alpha^2 e^{-2\alpha}) \\
    \label{eq:expansion_hphi1}
    =~& 2 \alpha^2 + 12 \alpha^2 e^{-\alpha} + \OO(\alpha^2 e^{-2\alpha}).
  \end{align}
  Similar calculations yield the expansions
  \begin{align}
    \label{eq:exp_dphi2}
    \phi'(2) =~& 2 \alpha e^{-\alpha} + 2 \alpha e^{-2\alpha} +
    \OO(\alpha e^{-3\alpha}), \quad \text{and} \\
    \label{eq:exp_hphi2}
    \phi''(2) =~& -2 \alpha^2 e^{-\alpha} + \OO(\alpha^2 e^{-3\alpha}).
  \end{align}
 
  Writing out $U^\scb_0$ in terms of the Morse potential, and using
  the fact that $2 \leq 4 - 2/h_0 \leq 4$, which ensures that
  $\phi''(1) + (4 - 2/h_0) \phi''(2) \geq W''(1) > 0$, we obtain
   \begin{align*}
     h_0 U^\scb_0 =~& \frac{\phi'(2)}{\phi''(1) + (4-2/h_0)\phi''(2)}
     = \frac{\phi'(2)}{\phi''(1)} \frac{1}{1 + (4-2/h_0)
       \smfrac{\phi''(2)}{\phi''(1)}} \\
     =~& \frac{\phi'(2)}{\phi''(1)} \B[ 1 - \b(4 - \smfrac{2}{h_0}\b)
     \smfrac{\phi''(2)}{\phi''(1)} + \OO\b(\b(\smfrac{\phi''(2)}{\phi''(1)}\b)^2\b)\B].
   \end{align*}
   Inserting the expansions \eqref{eq:expansion_hphi1} to
   \eqref{eq:exp_hphi2} gives \eqref{eq:asymp:scb}.

   To prove \eqref{eq:asymp:atm} we first expand $\lambda$ in terms of
   $\beta := \frac{\phi''(2)}{\phi''(1)}$, and then in terms of
   $e^{-\alpha}$,
   \begin{align}
     \notag
     \lambda =~& \smfrac{1}{2\beta} \B( \sqrt{1 + 4 \beta} - 1 - 2
     \beta \B) \\
     \notag
     =~& \smfrac{1}{2\beta} \B( 1 + \smfrac12 (4 \beta) - \smfrac18
     (4\beta)^2 + \smfrac1{16}(4\beta)^3 + \OO(\beta^4) - 1 - 2 \beta
     \B) \\
     \label{eq:asymp_lambda}
     =~& - \beta + 2 \beta^2 + \OO(\beta^3) 
     = e^{-\alpha} - 4 e^{-2\alpha} + \OO(e^{-3\alpha}).
   \end{align}
   Inserting this result into \eqref{eq:atm_soln} and a brief
   computation yield
   \begin{align*}
     u^\a_0 =~& \frac{\phi_2'}{\phi_1'' + \phi_2''(1+\lambda)} 
     = \smfrac{e^{-\alpha}}{\alpha} - 4 \smfrac{e^{-2\alpha}}{\alpha} + \OO\b(\smfrac{e^{-3\alpha}}{\alpha}\b).
   \end{align*}
   Since $u^\a_\x = u^\a_0 \lambda^{-\x}$ the result
   \eqref{eq:asymp:atm} follows easily.
\end{proof}

\bibliographystyle{plain}
\bibliography{scb}

\end{document}